\documentclass[11pt]{amsart}

\usepackage{amsmath, amsthm, amssymb, amsfonts, enumerate, graphicx, multicol, multirow}

\newtheorem{lem}{Lemma}[section]
\newtheorem{thm}[lem]{Theorem}
\newtheorem*{thma}{Theorem A}
\newtheorem*{thmb}{Theorem B}

\newtheorem{cor}[lem]{Corollary}

\theoremstyle{definition}
\newtheorem{dfn}[lem]{Definition}

\newtheorem{question}[lem]{Question}
\theoremstyle{remark}
\newtheorem{rmk}[lem]{Remark}

\DeclareMathOperator{\Ext}{Ext}

\DeclareMathOperator{\opH}{H}
\DeclareMathOperator{\Hom}{Hom}

\DeclareMathOperator{\Ind}{Ind}

\newcommand{\N}{\mathbb{N}}

\newcommand{\Z}{\mathbb{Z}}

\newcommand{\SL}{\mathrm{SL}}

\begin{document}

\title{On the cohomology of $\SL_2$ with coefficients in a simple module}
\author{John Rizkallah}
\address{Homerton College, Cambridge}
\email{jr588@cam.ac.uk}

\pagestyle{plain}

	\begin{abstract}
	Let $G$ be the simple algebraic group $\SL_2$ defined over an algebraically closed field $k$ of characteristic $p > 0$.
	Using results of A. Parker, we develop a method which gives, for any $q \in \N$, a closed form description of all simple modules $M$ such that $\opH^q(G,M) \neq 0$, together with the associated dimensions $\dim \opH^q(G,M)$.
	We apply this method for arbitrary primes $p$ and for $q \leq 3$, confirming results of Cline and Stewart along the way.
	Furthermore, we show that under the hypothesis $p > q$, the dimension of the cohomology $\opH^q(G,M)$ is at most 1, for any simple module $M$.
	Based on this evidence we discuss a conjecture for general semisimple algebraic groups.
	\end{abstract}

	\maketitle
	
	\section{Introduction}
	
	In \cite{PS1}, the authors prove that for any semisimple simply connected algebraic group $G$ with root system $\Phi$ over an algebraically closed field $k$ of positive characteristic, there is a constant $c = c(\Phi,q)$ such that $\dim\opH^q(G,L(\lambda)) \leq c$ for any dominant weight $\lambda$.
	In particular, this constant can be chosen independently of the characteristic $p$ of $k$.
	However, one cannot drop the dependence of $c$ on $q$, even for the case $G = \SL_2$: in \cite{DS2}, Stewart shows that, for any fixed $p$, the sequence
	\[ \{ \gamma_q \} = \{ \max_\lambda \dim \opH^q(\SL_2,L(\lambda)) \} \] grows exponentially in $q$.
	However, for $p$ sufficiently large compared to the degree of the cohomology, we show that in the case $G = \SL_2$ the constant $c$ can indeed be chosen independently of $q$.
	Specifically, we show that, if $p > q$, then $\dim \opH^q(\SL_2,L(\lambda)) \leq 1$.
	In order to prove this result we develop, using a theorem of Parker, a method for finding the weights $\lambda$ such that the space $\opH^q(\SL_2,L(\lambda))$ is non-trivial.
	This method also produces a closed-form description of these weights which is uniform in $p$.
	We demonstrate the method in the cases $q = 1, 2, 3$.
	The first case recovers a special case of a result of Cline in \cite{EC1}, the second case recovers a result of Stewart in \cite{DS1}.
	The third case is a new result:
	
	\begin{thma}
	Let $M = L(\lambda)^{[n]}$ be any Frobenius twist (possibly trivial) of the simple module of highest weight $\lambda$.
	If $\opH^3(\SL_2,M) \neq 0$ and $p > 2$ then $\lambda$ is one of
		\begin{enumerate}
		\item $4p - 2$
		\item $2p^2 - 4p$ \hspace{5mm} $(p > 3)$
		\item $p^n(2p-2) + 2p$ \hspace{5mm} $(n > 1)$
		\item $2p^2+2p-2$
		\item $p^n(2p-2)+2p^2-2p-2$ \hspace{5mm} $(n > 2)$
		\item $2p^3-2p^2-2p-2$
		\item $p^n(\lambda_2) + 2p-2$ \hspace{5mm} $(n > 1,$ $\lambda_2$ any 2-cohomological weight).
		\end{enumerate}
	Furthermore, $\dim\opH^3(\SL_2,M) = 1$.
		
		If $\opH^3(\SL_2,M)\neq 0$ and $p=2$ then $\lambda$ is one of
		\begin{enumerate}\item $6$
		\item $8$
		\item $2^n+2$ $(n>3)$
		\item $2^n+4$ $(n>3)$
		\item $2^n+10$ $(n>4)$.
		\end{enumerate}
	Furthermore, $\dim\opH^3(\SL_2,M) = 1$, except when $\lambda = 2^n + 4$ and $n > 4$, in which case $\dim\opH^3(\SL_2,M) = 2$.
		\end{thma}
	
	In \S\ref{s:genres}, we apply our method to get generic results when $p>q$ for $q$ the degree of cohomology. For this we use Parker's formula together with Jantzen's translation functors; ultimately, we get a closed form description of certain basic $q$-cohomological weights, a set we denote by $W_q$.
	
	In order to express the set $W_q$, we give the following definition, 
	
	\begin{dfn}\label{shifts}
	Given a weight $\lambda\in\N$ with $\lambda \equiv 0\mod 2p$ or $\lambda \equiv -2\mod 2p$, and $n \in \N$, define $\lambda$ \textit{shifted by $n$}, denoted $\lambda \parallel n$, as follows:
	\begin{enumerate}
	\item If $\lambda \equiv 0\mod 2p$ and $n$ is even, $\lambda \parallel n = p(\lambda + n)$.
	\item If $\lambda \equiv 0\mod 2p$ and $n$ is odd, $\lambda \parallel n = p(\lambda + n) + p - 2$.
	\item If $\lambda \equiv -2\mod 2p$ and $n$ is even, $\lambda \parallel n = p(\lambda - n)$.
	\item If $\lambda \equiv -2\mod 2p$ and $n$ is odd, $\lambda \parallel n = p(\lambda - n) + p - 2$.
	\end{enumerate}
	\end{dfn}

	The main result is
	
	\begin{thmb}
  	Assume $p > q$. Let $M=L(\lambda)^{[n]}$ be any Frobenius twist (possibly trivial) of the simple $\SL_2$-module of highest weight $\lambda$, and assume that $\dim\opH^q(\SL_2,M)\neq 0$. Then $\lambda$ lies in the set $W_q$, where \[W_q = \{ (p^n\lambda_i) \parallel (q-i) \mbox{ } | \mbox{ } n \geq 0, \lambda_i \in W_i, i = 0 \dots q-1\}.\]
	
	Moreover, $\dim \opH^q(SL_2,M)=1$.
	\end{thmb}
	
	We finish this section by asking if the result of Theorem B can be generalised.
	
		\begin{question}
	Does there exists a constant $p_0 = p_0(q)$ such that, if $p > p_0$ and $G$ is a semisimple algebraic group with root system $\Phi$
	over an algebraically closed field of	characteristic $p$,	then there is a constant $c = c(\Phi)$ such that $\dim\opH^q(G,M) \leq c$ for any simple $G$-module $M$?
	\end{question}
	
	\section{Notation and Preliminaries}
	
	Let $G = \SL_2$ defined over an algebraically closed field $k$ of characteristic $p > 0$.
	Choose a Borel subgroup $B$ corresponding to the negative root, and a maximal torus $T$ contained in $B$.
	The lattice of weights $X(T)$ may be identified with $\Z$, and the lattice of dominant weights $X(T)^+$ with $\N$.
	Thus for each positive integer $\lambda$ there is an irreducible module $L(\lambda)$ of highest weight $\lambda$.
	We also have the induced modules $\nabla(\lambda) := \opH^0(G/B,k_\lambda) = \Ind_B^G(k_\lambda)$, and the Weyl modules $\Delta(\lambda) = \nabla(-w_0\lambda)^*$, where $w_0$ is the longest element of the Weyl group.
	Let $V$ be a $G$-module.
	The $d$th Frobenius automorphism of $G$ raises each matrix entry to the power $p^d$.
	When composed with the representation $G \rightarrow GL(V)$, this induces the module $V^{[d]}$, called the $d$th Frobenius twist of $V$.
	
	For any G-module $M$, the functor $\Hom_G(M,-)$ of $G$-modules is left exact and so has right derived functors, which we denote by $\Ext_G^n(M,-)$.
	We define the cohomology functors $\opH^n(G,-)=\Ext_G^n(k,-)$.
	We will drop the subscript $G$ and simply write $\Ext^n$.
	
	The following three results are central to our analysis.
	
	\begin{thm}[Steinberg's Tensor Product Theorem]
	Let $\lambda \in \N$ be a dominant weight, and write $\lambda = \lambda_0 + p\lambda_1 + \dots + p^n\lambda_n$, where $0 \leq \lambda_i < p$ for each $i$. Then
	\[ L(\lambda) \cong L(\lambda_0) \otimes L(\lambda_1)^{[1]} \otimes \dots \otimes L(\lambda_n)^{[n]} \]
	as $G$-modules.
	\end{thm}
	
	\begin{thm}[{\cite[The Linkage Principle]{HA1}}] \nocite{JCJ}
	Let $\lambda$ and $\mu$ be dominant weights, and let $i \in \N$. Then
	\[ \Ext^i(\Delta(\mu),L(\lambda)) \neq 0 \Rightarrow \mbox{ either } \lambda + \mu \equiv -2\mod2p \mbox{ or } \lambda \equiv \mu\mod2p. \]
	\end{thm}
		
	\begin{thm}[{\cite[Theorem 6.1]{AP1}}]\label{parkersthm}
	
	Suppose that $p > 2$ is prime, $b \in \mathbb{N}$, $0 \leq i \leq p-2$, and $M$ is a finite-dimensional rational $G$-module. Then

	\begin{equation} \label{even}
	\Ext^q(\Delta(pb+i),M^{[1]} \otimes L(i)) \cong
	\displaystyle \bigoplus_{\substack{n=0 \\ n \text{ even}}}^{n=q} \Ext^{q-n}(\Delta(n+b), M)
	\end{equation}
	
	\begin{equation} \label{odd}
	\Ext^q(\Delta(pb+i),M^{[1]} \otimes L(p-2-i)) \cong
	\displaystyle \bigoplus_{\substack{n=0 \\ n \text{ odd}}}^{n=q} \Ext^{q-n}(\Delta(n+b), M)
	\end{equation}
	
	\end{thm}

We will also use the following, a consequence of the equivalence of categories of $G$-modules arising from the functor $F:V\mapsto \mathrm{St}_1\otimes V^{[1]}$ (see \cite[II.10.5]{JCJ} and \cite[\S3]{AP1}). We have

\begin{equation}\label{cancelsteinberg}\Ext^q(\Delta(pb+p-1,M^{[1]}\otimes L(p-1))\cong\Ext^q(\Delta(b),M)\end{equation}

	With the help of the Linkage Principle and the Tensor Product Theorem, these formulae can be used to find closed-form descriptions of the sets $W_q$ defined below.
	This is the method we develop in this paper.
		
	\begin{dfn}
	We say that a weight $\lambda$ is $(q,n)$-\textit{cohomological} if $\Ext^q(\Delta(n),L(\lambda)) \neq 0$.
	In particular, we say that $\lambda$ is $q$-\textit{cohomological} if it is $(q,0)$-cohomological, i.e. if $\Ext^q(\Delta(0),L(\lambda)) = \opH^q(\SL_2,L(\lambda)) \neq 0$.
	
	For each $q \geq 1$, the set of all weights $\lambda$ such that $\opH^q(\SL_2,L(\lambda)) \neq 0$ and $\opH^q(\SL_2,L(\lambda)^{[-1]}) = 0$ will be denoted $W_q$.
	We refer to the elements of $W_q$ as the \textit{maximally untwisted} $q$-cohomological weights.
	We will take $W_0 = \{ 0 \}$, as this makes the statement of Proposition \ref{wq} neater.
	\end{dfn}
	
	\section{Closed form descriptions of $\dim \opH^q(G,M)$ for fixed primes: Proof of Theorem A}
	
	We demonstrate an algorithm which uses Theorem \ref{parkersthm} to give, for fixed $p$, a list of the $(q,r)$-cohomological weights $\lambda$, and the dimensions of $\Ext^q(\Delta(r),L(\lambda))$, provided that
	\begin{enumerate}[(i)]
		\item all the $(s,t)$-cohomological weights and the associated dimensions are known for all $s < q$, where $s + t = q + r'$; and
		\item all the $(q,r')$-cohomological weights and the associated dimensions are known;
	\end{enumerate}
	where $r = pr' + r_0$ and $0 \leq r_0 < p$.
	
	Certainly the $(0,r)$-cohomological weights are known: if $\Ext^0(\Delta(r),L(\lambda)) \neq 0$ then $\lambda = r$, and the dimension of the $\Ext$-group is $1$.
	We aim to continue by induction, producing iteratively a list of all $q$-cohomological weights for any fixed $q \in \N$.
	
	Fix $q \in \N$ and let $\lambda, r$ be fixed but arbitrary with $\lambda = p\lambda' + \lambda_0$ and $r = pr' + r_0$.
	Assume that $\Ext^q(\Delta(r),L(\lambda)) \neq 0$ and also assume that (i) and (ii) hold.
	If $r_0 = p-1$ then by Equation (\ref{cancelsteinberg}) we have \[ \Ext^q(\Delta(r),L(\lambda)) = \Ext^q(\Delta(r'),L(\lambda')). \]
	Now by (ii) we know the non-zero dimensions of $\Ext^q(\Delta(r'),L(\lambda'))$, so we may write down a list of all the $(q,r)$-cohmological weights in this case.
	Specifically, suppose $\lambda'$ is a $(q,r')$-cohomological weight with associated dimension $d$;
	then $p\lambda' + p - 1$ is a $(q,r)$-cohomological weight with associated dimension $d$, and every $(q,r)$-cohomological weight arises this way.
	This deals with the case $r_0 = p-1$.
	
	Otherwise, $r_0 \leq p-2$. By the Linkage Principle, we have either $\lambda_0 = r_0$ or $\lambda_0 = p - 2 - r_0$.
	In the first case, Parker's formula (\ref{even}) gives us that \[ \Ext^q(\Delta(r),L(\lambda)) = \displaystyle \bigoplus_{\substack{n=0 \\ n \text{ even}}}^{n=q} \Ext^{q-n}(\Delta(n+r'), L(\lambda')).\]
	We analyse the direct summands in turn.
	The first direct summand is $\Ext^q(\Delta(r'),L(\lambda'))$.
	If $r > 0$ we have $r' < r$, and so again by (ii) we have that $\Ext^q(\Delta(r'),L(\lambda'))$ is known, so we may pass to the next summand.
	If $r = 0$, this first summand is $\Ext^q(\Delta(0),L(\lambda'))$, which we may assume is also known, this time by induction on $\lambda$, since $\Ext^q(\Delta(0),L(0))=0$ for any $q>0$. So we may pass to the next summand.
	The remaining summands are of the form $\Ext^{q-n}(\Delta(n+r'),L(\lambda'))$ with $n \neq 0$.
	Now by (i) these values are assumed to be known, so we are done.
	More specifically, one obtains all the $(q,r)$-cohomological weights from the union of the $(q,r')$-cohomological weights together with the $(q-n,n+r')$-cohomological weights, where $n$ is even.
	To calculate the associated dimension of a $(q,r)$-cohomological weight $\lambda$, we note how many times $\lambda'$ appears as a $(q-n,n+r')$-cohomological weight over all even values of $n$, which is achieved by induction on $\lambda$.
	
	The second case (where $\lambda_0 = p - 2 - r_0$) is very similar to the first, using formula (\ref{odd}) instead of formula (\ref{even}).
	
	By way of example, we use the above procedure to prove the following theorem.
	
	\begin{thm}
	Let $M = L(\lambda)^{[n]}$ be any Frobenius twist (possibly trivial) of the simple module of highest weight $\lambda$.
	\begin{enumerate}[(i)]
	\item If $\opH^1(\SL_2,M) \neq 0$ then $\lambda = 2p-2$.
	\item If $\opH^2(\SL_2,M) \neq 0$ then $\lambda$ is one of
		\begin{itemize}
		\item $2p$
		\item $2p^2 - 2p - 2$ \hspace{5mm} $(p > 2)$
		\item $p^n(2p-2)+2p-2$ \hspace{5mm} $(n > 1)$.
		\end{itemize}
	\item If $\opH^3(\SL_2,M) \neq 0$ and $p > 2$ then $\lambda$ is one of
		\begin{itemize}
		\item $4p - 2$
		\item $2p^2 - 4p$ \hspace{5mm} $(p > 3)$
		\item $p^n(2p-2) + 2p$ \hspace{5mm} $(n > 1)$
		\item $2p^2+2p-2$
		\item $p^n(2p-2)+2p^2-2p-2$ \hspace{5mm} $(n > 2)$
		\item $2p^3-2p^2-2p-2$
		\item $p^n(\lambda_2) + 2p-2$ \hspace{5mm} $(n > 1,$ $\lambda_2$ any 2-cohomolgical weight).
		\end{itemize}
		If $p = 2$ then $\lambda$ is one of
		\begin{itemize}
		\item 6
		\item 8
		\item $2^n + 2$ \hspace{5mm} $(n > 3)$
		\item $2^n + 4$ \hspace{5mm} $(n > 3)$
		\item $2^n + 10$ \hspace{5mm} $(n > 4)$
		\end{itemize}
	\end{enumerate}
	\end{thm}
	
	\begin{proof}
	\textit{(i)}
	Write $\lambda = p\lambda' + \lambda_0$.
	First let us assume that $\lambda_0 = p - 2$. Then 
	\[ \Ext^1(\Delta(0),L(\lambda)) = \Ext^0(\Delta(1),L(\lambda')). \]
	This is nonzero if and only if $\lambda' = 1$, and so 
	\[ \dim\Ext^1(\Delta(0),L(\lambda)) = 1 \mbox{ when } \lambda = 2p - 2. \]
	If we assume that $\lambda_0 = 0$ then we get 
	\[ \Ext^1(\Delta(0),L(\lambda)) = \Ext^1(\Delta(0),L(\lambda')). \]
	This tells us that $\lambda$ is 1-cohomological if and only if $\lambda'$ is 1-cohomological.
	Since $\lambda = p\lambda'$, this is just the statement that the twist of a cohomological weight is itself cohomological. \\
	By the Linkage Principle, if $\lambda_0$ is neither 0 nor $p-2$ then $\Ext^1(\Delta(0),L(\lambda))$ is zero.
	Thus we conclude that the only 1-cohomological weights are $p^n(2p - 2)$ for $n \geq 0$, i.e. $W_1 = \{ 2p - 2 \}$, and that in this case the cohomological dimension is 1.
		
	\textit{(ii)}
	Write $\lambda = p\lambda' + \lambda_0$.
	First let us assume that $\lambda_0 = 0$. Then 
	\[ \Ext^2(\Delta(0),L(\lambda)) = \Ext^2(\Delta(0),L(\lambda')) \oplus \Ext^0(\Delta(2),L(\lambda')). \]
	In order for the $\Ext^0$ term to be nonzero we must have $\lambda' = 2$.
	This forces the $\Ext^2$ term to be zero, and so we conclude that 
	\[ \dim\Ext^2(\Delta(0),L(\lambda)) = 1 \mbox{ when } \lambda = 2p. \]
	Now let's assume that $\lambda_0 = p - 2$. Then 
	\[ \Ext^2(\Delta(0),L(\lambda)) = \Ext^1(\Delta(1),L(\lambda')). \]
	We now need to determine the $(1,1)$-cohomological weights.
	So we write $\lambda' = p\lambda'' + \lambda'_0$, and first assume that $\lambda'_0 = p - 3$. Then 
	\[ \Ext^1(\Delta(1),L(\lambda')) = \Ext^0(\Delta(1),L(\lambda'')). \]
	This is nonzero if and only if $\lambda'' = 1$, and so 
	\[ \dim\Ext^2(\Delta(0),L(\lambda)) = 1 \mbox{ when } \lambda = 2p^2 - 2p - 2. \]
	If we assume that $\lambda'_0 = 1$, we get 
	\[ \Ext^1(\Delta(1),L(\lambda')) = \Ext^1(\Delta(0),L(\lambda'')). \]
	This is nonzero if and only if $\lambda'' = p^n(2p-2)$, and so
	\[ \dim\Ext^2(\Delta(0),L(\lambda)) = 1 \mbox{ when } \lambda = 2p - 2 + p^n(2p-2) \mbox{ for some } n \geq 2. \]
	So $W_2 = \{ 2p, 2p^2 - 2p - 2, 2p - 2 + p^n(2p-2)$ $|$ $n \geq 2 \}$.

	\textit{(iii)}
	Write $\lambda = p\lambda' + \lambda_0$.
	First let us assume that $\lambda_0 = 0$. Then 
	\[ \Ext^3(\Delta(0),L(\lambda)) = \Ext^3(\Delta(0),L(\lambda')) \oplus \Ext^1(\Delta(2),L(\lambda')). \]
	Write $\lambda' = p\lambda'' + \lambda'_0$ and assume $\lambda'_0 = 2$. Then
	\[ \Ext^1(\Delta(2),L(\lambda')) = \Ext^1(\Delta(0),L(\lambda'')). \]
	So $\lambda'' = p^n(2p-2)$, giving $\lambda = p^n(2p-2) + 2p$ for some $n \geq 2$. \\
	If we assume that $\lambda'_0 = p - 4$ then
	\[ \Ext^1(\Delta(2),L(\lambda')) = \Ext^0(\Delta(1),L(\lambda'')) \]
	which gives $\lambda =  2p^2 - 4p$. \\
	Now assume that $\lambda_0 = p-2$. Then
	\[ \Ext^3(\Delta(0),L(\lambda)) = \Ext^2(\Delta(1),L(\lambda')) \oplus \Ext^0(\Delta(3),L(\lambda')). \]
	The $\Ext^0$ term is nonzero if and only if $\lambda'=3$, giving $\lambda = 4p-2$. \\
	For the $\Ext^2$ term, write $\lambda' = p\lambda'' + \lambda'_0$ and assume $\lambda'_0=1$. Then
	\[ \Ext^2(\Delta(1),L(\lambda')) = \Ext^2(\Delta(0),L(\lambda'')) \oplus \Ext^0(\Delta(2),L(\lambda'')). \]
	The $\Ext^0$ term is nonzero if and only if $\lambda''=2$, giving $\lambda = 2p^2+2p-2$.
	The $\Ext^2$ term is nonzero if and only if $\lambda=p^2(\lambda_2)+2p-2$, where $\lambda_2$ denotes any 2-cohomological weight. \\
	If we assume that $\lambda'_0=p-3$ then
	\[ \Ext^2(\Delta(1),L(\lambda')) = \Ext^1(\Delta(1),L(\lambda'')). \]
	So now we have to write $\lambda'' = p\lambda''' + \lambda''_0$. If $\lambda''_0=1$ then
	\[ \Ext^1(\Delta(1),L(\lambda'')) = \Ext^1(\Delta(0),L(\lambda''')) \]
	which is nonzero if and only if $\lambda'''=p^n(2p-2)$, giving $\lambda = p^n(2p-2)+2p^2-2p-2$ for some $n \geq 3$. \\
	If $\lambda''_0=p-3$ then
	\[ \Ext^1(\Delta(1),L(\lambda'')) = \Ext^0(\Delta(1),L(\lambda''')) \]
	which is nonzero if and only if $\lambda'''=1$, giving $\lambda = 2p^3-2p^2-2p-2$.
	
	For the $p=2$ result, we use the following formula, taken from \cite{AP1}:
	\[ \Ext^q(\Delta(2b),L(\lambda)) = \displaystyle \bigoplus_{n=0}^{n=q} \Ext^{q-n}(\Delta(b+n), L(\lambda')) \]
	where $\lambda = 2\lambda'$.
	So with $b=0$ and $q=3$, we have \[ \Ext^3(\Delta(0),L(\lambda)) = \Ext^3(\Delta(0),L(\lambda')) \oplus \Ext^2(\Delta(1),L(\lambda')) \oplus \Ext^1(\Delta(2),L(\lambda')) \oplus \Ext^0(\Delta(3),L(\lambda')). \]
	The $\Ext^0$ term is nonzero if and only if $\lambda'=3$, giving $\lambda = 6$. \\
	Applying the formula to the $\Ext^1$ term gives \[ \Ext^1(\Delta(2),L(\lambda')) = \Ext^1(\Delta(1),L(\lambda'')) \oplus \Ext^0(\Delta(2),L(\lambda'')). \]
	The $\Ext^0$ term is nonzero if and only if $\lambda''=2$, giving $\lambda = 8$. \\
	We have $\Ext^1(\Delta(1),L(\lambda'')) = \Ext^1(\Delta(0),L(\lambda'''))$, where $\lambda'' = 2\lambda''' + 1$.
	Taking $\lambda''' = 2^n$ for $n > 0$ gives $\lambda = 2^n + 4$ for $n > 3$. \\
	The next summand to analyse is $\Ext^2(\Delta(1),L(\lambda'))$.
	We have \[ \Ext^2(\Delta(1),L(\lambda')) = \Ext^2(\Delta(0),L(\lambda'')) \] where $\lambda' = 2\lambda'' + 1$.
	If this is non-zero then either $\lambda'' = 2^n$ with $n > 1$ or $\lambda'' = 2^n + 2$ with $n > 2$.
	These give $\lambda = 2^n + 2$ with $n>3$ and $\lambda = 2^n + 10$ with $n>4$ respectively.
	
	\end{proof}
	
	The first part of this theorem confirms a result of Cline in \cite{EC1}.
	The second part confirms a result of Stewart in \cite{DS1}.
	
	\section{Generic results for large primes: Proof of Theorem B}\label{s:genres}
	
	From the examples in the previous section, one sees that when $p \geq 5$, the list of cohomological weights is uniform with $p$.
	Indeed, by inspecting examples, one may observe that when $p > q$, the list of $q$-cohomological weights is uniform with $p$.
	In this section we will prove that this is always true.
	Henceforth we shall assume that $p > q$.
	
	Suppose $\lambda$ is $q$-cohomological, so $H = \opH^q(G,L(\lambda)) = \Ext^q(\Delta(0),L(\lambda)) \neq 0$.
	Then it follows from the Linkage Principle that either $\lambda = p\lambda'$ for some even $\lambda'$ or $\lambda = p\lambda' + p - 2$ for some odd $\lambda'$.
	Thus
	\[ H = \displaystyle \bigoplus_{n=0}^{n=q} \Ext^{q-n}(\Delta(n), L(\lambda')) \]
	where the sum is taken over either even or odd numbers only.
	Clearly if $H$ is nonzero then one of the summands $\Ext^{q-i}(\Delta(i),L(\lambda'))$ is nonzero.
	So, Parker's formulae have reduced the computation of $q$-cohomological weights to the computation of $(q-n,n)$-cohomological weights, for all $0 < n \leq q$.
	Furthermore, every maximally untwisted $q$-cohomological weight arises from a $(q-n,n)$-cohomological weight in this way.
	
	The next result shows how the $(q-n,n)$-cohomological weights themselves arise from $(q-n)$-cohomological weights.
	
	\begin{lem}\label{lemon}
	Suppose $\lambda \in X(T)^+$ is linked to zero, and let $n \leq p - 2$ be an integer.
	\begin{enumerate}[(i)]
	\item If $\lambda \equiv 0\mod 2p$ then $\Ext^i(\Delta(0),L(\lambda)) = \Ext^i(\Delta(n),L(\lambda+n))$.
	\item If $\lambda \equiv -2\mod 2p$ then $\Ext^i(\Delta(0),L(\lambda)) = \Ext^i(\Delta(n),L(\lambda-n))$.
	\end{enumerate}
	\end{lem}
	
	\begin{proof}
	In \cite[II.7.6]{JCJ}, Jantzen gives the isomorphisms
	\[ \Ext^i_G(V,T^\lambda_\mu V') \cong \Ext^i_G(T^\mu_\lambda V,V') \]
	for all $i$, where $V$ and $V'$ are $G$-modules and $T^\mu_\lambda$ is a translation functor.
	Setting $\mu = \lambda + n$ we can apply this to get
	\begin{align*}
	\Ext^i(\Delta(0),L(\lambda)) &= \Ext^i(\Delta(0),T^\lambda_\mu L(\mu)) \\
	&\cong \Ext^i(T^\mu_\lambda \Delta(0),L(\mu)) \\
	&= \Ext^i(\Delta(n),L(\mu))
	\end{align*}
	which proves \textit{(i)}. Part \textit{(ii)} is similar.
	\end{proof}

	\begin{rmk}
	Alternatively, this Lemma can be proved directly from Parker's formulae:
	For part \textit{(i)}, if we expand Ext$^i(\Delta(0),L(\lambda))$ and Ext$^i(\Delta(n),L(\lambda+n))$ using formula (\ref{even}),
	we see that they are both equal to 
	\[ \Ext^i(\Delta(0),L(\lambda)^{[-1]}) \oplus \Ext^{i-2}(\Delta(2),L(\lambda)^{[-1]}) \oplus \dots \oplus \Ext^a(\Delta(i-a),L(\lambda)^{[-1]}) \]
	where $a=1$ if $i$ is odd, $a=0$ otherwise. Part \textit{(ii)} is proved similarly with formula (\ref{odd}).
	\end{rmk}
	
	Therefore, all $q$-cohomological weights arise from $(q-n)$-cohomological weights via the method described above.
	
	Lemma \ref{lemon} motivates Definition \ref{shifts} which the reader should recall now.
	
	\begin{cor}\label{core}
	If $\lambda$ is $(q-n)$-cohomological, then $\lambda \parallel n$ is $q$-cohomological.
	\end{cor}
	
	\begin{proof}
	There are four separate cases to consider, which, following Definition \ref{shifts}, arise from the parity of $q-n$ and the parity of $\lambda$.
	We will only prove the case where $q-n$ is even and $\lambda \equiv 0 \mod 2p$; the remaining cases can be proved with similar arguments.
	 We have:
	\begin{align*}
	\dim\Ext^q(\Delta(0),L(\lambda \parallel n)) &= \dim\Ext^q(\Delta(0),L(p(\lambda + n))) \\
	&\geq \dim\Ext^{q-n}(\Delta(n),L(\lambda + n)) &( \mbox{by formula (\ref{even})}) \\
	&= \dim\Ext^{q-n}(\Delta(0),L(\lambda)) &( \mbox{by Lemma \ref{lemon}}) \\
	&> 0. &( \mbox{by assumption})
	\end{align*}
	\end{proof}
	
	\begin{lem}\label{twisting}
	Assume $p > q \geq 1$.
	If $\lambda$ is linked to zero then $\opH^q(G,L(\lambda)) = \opH^q(G,L(\lambda)^{[1]})$.
	\end{lem}
	
	\begin{proof}We have
	\begin{align*}
	\opH^q(G,L(\lambda)^{[1]}) &= \Ext^q(\Delta(0),L(\lambda)^{[1]}) \\
	&= \displaystyle \bigoplus_{\substack{n=0 \\ n \text{ even}}}^{n=q} \Ext^{q-n}(\Delta(n),L(\lambda)) &\mbox{ (by formula (\ref{even}))} \\
	&= \Ext^q(\Delta(0),L(\lambda)) &\mbox{ (by the Linkage Principle, since $n \leq q < p$)}\\
	&= \opH^q(G,L(\lambda)).
	\end{align*}
	\end{proof}
	
	The following is an immediate consequence of the lemma, using the injective map $\opH^m(G,M)\to \opH^m(G,M^{[1]})$ induced by the Frobenius twist.
	
	\begin{cor}\label{lemur}
	If $\lambda$ is $q$-cohomological then $p^n\lambda$ is $q$-cohomological for all $n > 0$.
	\end{cor}
	
	The next result gives a closed-form description of the set $W_q$.
		
	\begin{thm}\label{wq}
	Assume $p > q \geq 1$.
	Then \[ W_q = \{ (p^n\lambda_i) \parallel (q-i) \mbox{ } | \mbox{ } n \geq 0, \lambda_i \in W_i, i = 0 \dots q-1 \}. \]
	\end{thm}
	
	\begin{proof}
	If $\lambda \in W_q$ then by Theorem \ref{parkersthm} there is an integer $n$ such that $1 \leq n \leq q$ and the space $\Ext^{q-n}(\Delta(n),L(\lambda'))$ is nonzero (where $\lambda = p\lambda' + \lambda_0$ and $0 \leq \lambda_0 < p$).
	By Lemma \ref{lemon}, either $\lambda' + n$ or $\lambda' - n$ is $(q-n)$-cohomological.
	But then either $\lambda = (\lambda' + n) \parallel n$ or $\lambda = (\lambda' - n) \parallel n$,
	and so $\lambda \in \{ (p^n\lambda_i) \parallel (q-i) \mbox{ } | \mbox{ } n \geq 0, \lambda_i \in W_i, i = 0 \dots q-1 \}$.
	
	It then follows from Corollaries \ref{lemur} and \ref{core} that $W_q$ contains this set.
	\end{proof}
	
	\begin{proof}[Proof of Theorem A]
	If $\opH^q(G,L(\lambda)) \neq 0$ then by Lemma \ref{twisting} we may assume that $\lambda \in W_q$.	
	So it remains to show that if $\lambda \in W_q$ then $\dim\opH^q(G,L(\lambda)) = \dim\Ext^q(\Delta(0),L(\lambda)) = 1$.
	We will use induction on $q$, and restrict our attention to the case where $q-i$ is even, $\lambda_i \equiv 0 \mod 2p$ and $n=0$; as before, the remaining cases are dealt with similarly.
	Let $\lambda_i \in W_i$, so that $\lambda = \lambda_i \parallel (q-i) \in W_q$.
	Then \[ \Ext^q(\Delta(0),L(\lambda)) \cong \displaystyle \bigoplus_{\substack{n=0 \\ n \text{ even}}}^{n=q} \Ext^{q-n}(\Delta(n),L(\lambda_i + q-i)). \]
	Since $n \leq q < p$, then by the Linkage Principle, the only term that can be nonzero is $n=q-i$. Thus
	\begin{align*}
	\Ext^q(\Delta(0),L(\lambda)) &= \Ext^i(\Delta(q-i),L(\lambda_i + q-i)) \\
	&= \Ext^i(\Delta(0),L(\lambda_i)) &( \mbox{by Lemma \ref{lemon}})
	\end{align*}
	whose dimension is 1 by induction.
	\end{proof}

	\subsection*{Acknowledgements}
	I would like to thank my supervisor D. Stewart for his help in producing this paper.
		
	\bibliographystyle{amsalpha}
	\bibliography{bib}

\providecommand{\bysame}{\leavevmode\hbox to3em{\hrulefill}\thinspace}
\providecommand{\MR}{\relax\ifhmode\unskip\space\fi MR }
\providecommand{\MRhref}[2]{%
  \href{http://www.ams.org/mathscinet-getitem?mr=#1}{#2}
}
\providecommand{\href}[2]{#2}
\begin{thebibliography}{And80}

\bibitem[And80]{HA1}
Henning~Haahr Andersen, \emph{The strong linkage principle}, J. reine angew.
  Math. \textbf{315} (1980), 53--59.

\bibitem[Cli79]{EC1}
Edward Cline, \emph{Ext$^1$ for {SL}$_2$}, Comm. Algebra \textbf{7} (1979),
  107--111.

\bibitem[Jan03]{JCJ}
Jens~Carsten Jantzen, \emph{Representations of {A}lgebraic {G}roups}, Amer.
  Math. Soc., 2003.

\bibitem[Par07]{AP1}
Alison Parker, \emph{Higher extensions between modules for ${SL}_2$}, Adv.
  Math. \textbf{209} (2007), 381--405.

\bibitem[PS11]{PS1}
Brian Parshall and Leonard Scott, \emph{{Bounding Ext for Modules for Algebraic
  Groups, Finite Groups and Quantum Groups}}, Adv. Math. \textbf{226} (2011),
  2065--2088.

\bibitem[Ste10]{DS1}
David Stewart, \emph{{The Second Cohomology of Simple $SL_2$-modules}}, Proc.
  Amer. Math. Soc. \textbf{138} (2010), 427--434.

\bibitem[Ste12]{DS2}
\bysame, \emph{Unbounding {E}xt}, J. Algebra \textbf{365} (2012), 1--11.

\end{thebibliography}
	
\end{document}